\documentclass{article}

\input{style-preprint}

\numberwithin{equation}{section}

\newcommand{\midb}{\;\middle|\;}

\begin{document}

\title{The Scaling Limit of Superreplication Prices with Small Transaction Costs in the Multivariate Case}

\author{Peter Bank\thanks{TU Berlin, Dept.\ of Mathematics, \textit{bank@math.tu-berlin.de}. Corresponding author.} \and 
        Yan Dolinsky\thanks{Hebrew University,  Dept.\ of Statistics, \textit{yan.dolinsky@mail.huji.ac.il}} \and
		Ari-Pekka Perkki\"o\thanks{TU Berlin, Dept.\ of Mathematics, \textit{perkkioe@math.tu-berlin.de}}\ \thanks{All authors are grateful to the Einstein Foundation for the financial support through its research project on ``Game options and markets with frictions''.}}


\date{\today}

\maketitle

\begin{abstract}
Kusuoka [{\em Limit Theorem on Option Replication Cost with Transaction Costs,} Ann. Appl. Probab. {\bf 5}, 198--221, (1995).] showed how to obtain non-trivial scaling limits of superreplication prices in discrete-time models of a single risky asset which is traded at properly scaled proportional transaction costs. This article extends the result to a multi-variate setup where the investor can trade in several risky assets. The $G$-expectation describing the limiting price involves models with a volatility range around the frictionless scaling limit that depends not only on the transaction costs coefficients but also on the chosen complete discrete-time reference model.
\end{abstract}


\section{Introduction}\label{intro}

It is well-known that in typical continuous-time models with proportional transaction costs the super-replication costs of a contingent claim become prohibitively high and often lead to practically rather unhelpful buy-and-hold hedging strategies; see \cite{DC} and \cite{BT,CPT,GRS,JLR,LS,SSC}). Roughly speaking, these high super-replication costs are due to the fact that the hedge has to stay on the ``safe side'' even in those (typically highly unlikely, but still relevant) scenarios where the underlying's price exhibits extreme movements to which the hedge cannot adjust fast enough under these market frictions. Of course, in simple discrete-time models such as the Cox--Ross--Rubinstein (CRR) binomial model of \cite{CRR} such movements are not possible and, in fact, super-replication becomes a more palatable concept there, even under proportional transaction costs. Yet, if, to improve tractability, one rescales the price-dynamics to obtain a continuous--time model such as the geometric Brownian motion of Black and Scholes, one finds oneself in the same predicament as in the continuous--time setting discussed above.

In his seminal paper \cite{K} Kusuoka showed, however, that if, along with the price dynamics, one also rescales the transaction costs in a CRR model, a meaningful scaling limit of super-replication prices emerges that can be described as the value of a volatility control problem on Wiener space, nowadays known as Peng's $G$-expectation (\cite{P1,P2}). 

The main result of the present paper is the extension of Kusuoka's results to a multivariate setting where the investor can trade in a number of risky assets. As a substitute for the one dimensional CRR model, we use He's \cite{H} multivariate models of a complete frictionless financial market. We introduce rescaled transaction costs and consider the super-replication prices of a contingent claim with continuous, but possibly path-dependent payoff exhibiting polynomial growth. Theorem~\ref{thm2.1} describes the scaling limit of these prices by the payoff's $G$-expectation under a family of probabilities. In line with Kusuoka's one dimensional result, this family is given in terms of a collection of stochastic volatility models where volatility is constrained to lie in a certain neighborhood of the (constant) reference volatility of the scaling limit in the CRR-model. Interestingly, this neighborhood depends not only on the size of transaction costs coefficients, but also on the choice of multivariate frictionless reference model.

\section{The problem setup and the main result}\label{sec:2}

\subsection{The reference complete market model and proportional transaction costs}

As in the seminal approach of~\cite{K}, we will work with a discrete-time complete market model and take its scaling limit. In a setting with several assets, a suitable construction of such a model is given in~\cite{H} where asset prices are driven by a random walk with i.i.d. increments $\xi_1, \xi_2,\dots$, uniformly distributed under the probability $\mathbb{P}$ on $(\Omega,\mathcal{F})$ on $d+1$ possible values $v_1,\dots, v_{d+1} \in \mathbb{R}^{d}$. The row vectors $v_1,\dots,v_{d+1}$ are chosen as vertices of a regular simplex so that they are affinely independent with
\begin{align}
|v_i|&=d,\quad \langle v_i,v_j\rangle=-1,\quad i\ne j,\label{2.0+}\\
 \sum_{j=1}^{d+1} v_j&=0,\label{2.1+}\\
\sum_{j=1}^{d+1} v^{'}_j v_j&=(d+1) I,\label{2.2+}
\end{align}
where $\langle\cdot,\cdot\rangle$ denotes the standard scalar product of $\mathbb{R}^d$ and $I$ denotes the identity matrix.

We assume the availability of a bank account which, for simplicity, bears no interest. Asset price dynamics $S^{(n)}$ in an $n$-period model will be defined using a $d\times d$-dimensional regular volatility matrix $\sigma$ which specifies price sensitivities with respect to the drivers $\xi_1,\xi_2,\dots$:
\begin{align}\label{2.3}
\begin{split}
S^{n,i}_k \set s_i \prod_{m=1}^k\left(1+
\sqrt{\frac{1}{n}}\langle \sigma_i,
\xi_m\rangle\right), \quad k=1,\dots,n,
\end{split}
\end{align}
where $s=(s_1,...,s_d)\in (0,\infty)^d$ denotes the asset prices at time 0. 

The regularity of the volatility matrix $\sigma$ ensures that the filtration
\[
\mathcal{F}_k \set \sigma\{\xi_k| m = 1,\dots, k\}, \quad k=0,1,\dots,n,
\]
coincides with the one generated by the asset prices. In conjunction with the affine independence of $v_1,\dots,v_{d+1}$ and~\eqref{2.1+}, it also implies that the measure $\mathbb{P}$ is the only martingale measure for $S^{(n)}$ on $\mathcal{F}_n$.

We assume that, when changing his position in the risky assets, the investor has to trade through his bank account, i.e., bartering directly between the risky assets is not possible. Transaction costs are then described by constants $\kappa^+_i, \kappa^-_i \geq 0$ measuring what fraction of a purchase or sale, respectively, of asset number $i=1,\dots,d$ incurs as transaction costs. Specifically, if $\pi=(x,\gamma)$ describes the investor's strategy where $x$ is
his initial net worth in cash and where $\gamma:=\{(\gamma^{(1)}_k,...,\gamma^{(d)}_k)\}_{k=0,\dots,n-1}$ is an adapted process with values in $\mathbb{R}^d$ describing the investor's holdings in each asset, his net worth in cash will end up as
\begin{align*}
Z^{\pi}_n  =x&+\sum_{j=1}^d \left(
(1-\kappa^{-}_j/\sqrt{n})S^{n,j}_n\max(0,\gamma^{(j)}_{n-1})\right.\\&\qquad\qquad\qquad\left.-
(1+\kappa^{+}_j/\sqrt{n})S^{n,j}_n\max(0,-\gamma^{(j)}_{n-1})
\right)\\ &+
\sum_{j=1}^d\sum_{i=0}^{n-1}\left((1-\kappa^{-}_j/\sqrt{n})S^{n,j}_i\max(0,\gamma^{(j)}_{i-1}-\gamma^{(j)}_{i})
            \right.\\&\qquad\qquad\qquad\left.- 
(1+\kappa^{+}_j/\sqrt{n})S^{n,j}_i\max(0,\gamma^{(j)}_{i}-\gamma^{(j)}_{i-1})\right),
\end{align*}
where we set $\gamma_{-1}\set 0$.

\subsection{Superreplication prices and their scaling limits}

The main result of this paper is the description of the scaling limit of superreplication prices in the above model. The derivatives we cover are given by continuous functionals $F$ on the space $C_d$ of continuous paths $f:[0,1] \to \mathbb{R}^d$ which satisfy a polynomial growth condition: 
\begin{align}\label{2.5}
\begin{split}
|F(f)|\leq \const (1+||f||^p) \quad (f\in C_d)
\end{split}
\end{align}
with the sup norm $||f||=\max_{t\in [0,1]} |f_t|$.

Via the interpolation operator $\mathcal{W}_n:({\mathbb R^d})^{n+1}\rightarrow {C}_d$ with
\begin{align}\label{2.6}
[\mathcal{W}_n(\{y_k\}_{k=0}^n)]_t=([nt]+1-nt)y_{[nt]}+(nt-[nt])y_{[nt]+1}, \quad t\in[0,1],
\end{align}
where $[z]$ denotes the integer part of $z$, such a functional $F$ induces, for $n=1,2,\dots$, the path dependent payoff functionals
\[
F_n \set F(\mathcal{W}_n(S^{(n)}))
\]
on $(\Omega,\mathcal{F}_n)$. Their superreplication prices are given by
\begin{align}\label{2.8}
V_n(F)=\inf\{x \in \mathbb{R}\;|\; Z^{\pi}_n \geq F_n \text{ $\mathbb{P}$-a.s. for some strategy $\pi=(x,\gamma)$}\}.
\end{align}

The main result of this paper is the identification of the scaling
limit $\lim_n V_n$ of these superreplication prices as a
$G$-expectation in the sense of Peng \cite{P1,P2}. This $G$-expectation
involves a family $\mathcal{Q}$ of martingale measures for the
coordinate process $\mathbb{B}$ on $(C_d,\mathcal{B}(C_d))$ for which
volatility takes values in a certain set $\Gamma$:
\begin{align}\label{eq:55}
\mathcal{Q} \set \{\mathbb{Q} \;:\; \text{$\BB$ is a $\QQ$-martingale with $\mathbb{B}_0 =0$, $d\langle \mathbb B\rangle/dt \in \Gamma$ $\mathbb{Q}$-a.s.}\}.
\end{align}
The set $\Gamma$ of allowed volatilities is the set of all symmetric $d\times d$-matrices 
\[
\sigma\sigma^{'}+\sigma \beta+\beta'\sigma'
\]
where $\beta$ is a $d\times d$-matrix from the set
\begin{align}\label{eq:1}
B \set \left\{\left[\beta_1\dots\beta_d\right]\midb \beta_{k}=\sum_{j=1}^{d+1} w_{jk} v_j'
\text{ for some } w_{jk} \in \left[0,\frac{\kappa_k^++\kappa_k^-}{d+1}\right]\right\}.
\end{align}

A technical assumption for our main result to hold is
\begin{asm}\label{asm2.1}
For any matrix $\beta \in B$, we have that $\sigma'+\beta$ is invertible with
\begin{align}\label{2.9}
v_i \beta (\sigma'+\beta)^{-1}v'_j>-1, \quad i,j=1,\dots,d+1,
\end{align}
\end{asm}
Observe that~\eqref{2.9} holds true for sufficiently small
$\kappa^{+}_1,...,\kappa^{+}_d,\kappa^{-}_1,...,\kappa^{-}_d$. We
refer to Section~\ref{sec6} below for a more detailed discussion of
this assumption, including an example that shows it cannot be omitted
in our main result:

\begin{thm}\label{thm2.1}
Let Assumption~\ref{asm2.1} hold and consider a European option whose payoff is described by a continuous functional $F:C_d \to \RR$ with polynomial growth as in~\eqref{2.5}. Then the scaling limit of super-replication prices with proportional transaction costs $\kappa^\pm/\sqrt{n}$ is
\begin{align}\label{eq:155}
\lim_{n\rightarrow\infty} V_n(F)=
\sup_{\mathbb Q\in\mathcal Q} \mathbb{E}_{\mathbb{Q}} [F(\mathbb S)].
\end{align}
where, under $\mathbb{Q} \in \mathcal{Q}$, $\mathbb{S}$ denotes the vector of Dol{\'e}ans-Dade exponentials
\begin{align}\label{eq:3}
\mathbb{S}^{(i)}_t \set s_i
\exp\left(\mathbb{B}^{(i)}_t-\frac{1}{2}\langle\mathbb{B}^{(i)}\rangle_t\right),\quad t \in [0,1], \quad i=1,\dots,d.
\end{align}
\end{thm}

Hence, as in the one dimensional setting of Kusuoka~\cite{K}, the
scaling limit of superreplication prices is a $G$-expectation in the
sense of Peng~\cite{P1,P2}. The set $\cQ$ of models considered for
this $G$-expectation is parametrized by volatility processes taking
values in a neighborhood of the multi-variate geometric Brownian
motion. Remarkably, though, the geometry of this neighborhood not only
depends on the different transaction cost coefficients
$\kappa^{\pm}_1, \dots, \kappa^{\pm}_d$, but also on the choice of
possible directions $v_1,\dots,v_{d+1}$ in which the discrete-time
drivers can move. We give an example in Section~\ref{sec6}
below. 

Clearly, one could think of using other discrete-time reference models
than those of \cite{H} that we focussed on. One should bear in mind
though that these ought to be complete to ensure the Black-Scholes
model as the non-trivial scaling limit in the frictionless case.
Indeed, scaling limits of super-replication prices in incomplete
discrete-time models can become trivial in the scaling limit as we
illustrate in Section~\ref{sec6} for the case of a multivariate
Cox-Ross-Rubinstein model.

\section{Proof of the main result}

As in~\cite{K}, the starting point for the proof of
Theorem~\ref{thm2.1} is the well-known dual description of
discrete-time superreplication prices via consistent price systems as,
e.g., in Theorem~3.2 in~\cite{JK}. For any $n=1,2,\dots$, we denote by
$\mathcal Q_n$ the set of all probability measures $\mathbb Q_n$
absolutely continuous with respect to $\mathbb P$ on $(\Omega,\cF_n)$ such that there exists a
$d$--dimensional $\mathbb Q_n$-martingale
$M^{(n)}=\{M^{n,1}_k,...,M^{n,d}_k\}_{k=0,\dots,n}$ (with respect
to $\{\mathcal F_k\}_{k=0,\dots,n}$) whose components satisfy
$\mathbb P$-a.s.
\begin{align}\label{3.100}
\left(1-\frac{\kappa^{-}_i}{\sqrt n}\right)S^{n,i}_k\leq M^{n,i}_k\leq
\left(1+\frac{\kappa^{+}_i}{\sqrt n}\right)S^{n,i}_k, \quad
k=0,\dots,n.
\end{align}

\begin{lem}\label{lem3.1}
We have
\[
V_n(F)=\sup_{\mathbb Q_n\in\mathcal Q_n} \mathbb E_{\mathbb Q_n}
\left[F\left(\mathcal{W}_n(S^{(n)})\right)\right].
\]
\end{lem}

The following lemma will allow us to cover the case of polynomially
growing continuous payoffs $F$, rather than merely bounded ones. 

\begin{lem}\label{lem:coe}
  Assume that $F$ and $p$ are as in \eqref{2.5} and that
  $\Law(Z^{(n)} \mid \QQ_n ) \Rightarrow \Law(Z \mid \QQ)$ for some
  $C_d$-valued random variables $Z^{(n)},Z$ on probability spaces
  $(\Omega^{(n)},\mathcal F^{(n)},\QQ^{(n)})$,
  $(\Omega,{\mathcal F},\QQ)$, respectively. If
  $\sup_n \E_{\QQ^{(n)}}[(\max_t Z^{(n)}_t)^m] <\infty$ for some $m>p$, then
\[
\lim_{n\rightarrow\infty} \E_{\QQ^{(n)}}[F(Z^{(n)})]= E_\QQ [F(Z)].
\]
In particular, $\QQ\mapsto E_\QQ F(\mathbb S)$ is weakly continuous on $\mathcal Q$.
\end{lem}
\begin{proof}
By Skorokhod representation theorem, there exists a probability space $(\hat\Omega,\hat {\mathcal F},\hat{\mathbb P})$ and $\hat Z^{(n)},\hat Z$ on $\hat {\Omega}$ such that $\hat Z^{(n)}\rightarrow\hat Z$ almost surely and such that the laws of $\hat Z^{(n)}$ and $\hat Z$ coincide with those of $Z^{(n)}$ and $Z$. The growth condition of $F$ and existence of $m>p$ imply that $(F(\hat Z^{(n)}))_{n=1}^\infty$ forms a uniformly integrable family which gives the first claim.

As to second claim, assume that $\QQ^n\Rightarrow \QQ$ for some martingales laws from $\mathcal Q$. By the example after Theorem~4.4 in~\cite{DP}, $\Law(\mathbb S\mid \QQ^n))\Rightarrow\Law(\mathbb S,\QQ)$. Using Burkholder--Davis--Gundy inequality and the compactness of $\Gamma$, the claim now follows from the first part.
\end{proof}

In the next section, using Lemma~\ref{lem:coe}, we will get that 
\begin{align}\label{2.lower} 
\limsup_{n\rightarrow\infty} V_n(F)\leq
\sup_{\mathbb Q\in\mathcal Q} \mathbb{E}_{\mathbb{Q}} [F(\mathbb S)]
\end{align} 
provided that the set $\mathcal{Q}$ of~\eqref{eq:55} contains all weak limits of sequences the form $(\Law(\cW_n(\tilde M^{(n)}) \;|\; \QQ_n))_{n=1}^\infty$ for arbitrary choices of $\QQ_n \in \cQ_n$. This is accomplished even without Assumption~\ref{asm2.1} which is only needed in Section~\ref{sec5}, where we show that $\cQ$ is not ``too big''. For this, we extend, to our multi-dimensional setting, Kusuoka's construction of suitable discrete-time approximations 
to deduce
\begin{align}\label{2.upper} 
\liminf_n V_n(F) \geq \sup_{\mathbb Q\in\mathcal Q}
\mathbb{E}_{\mathbb{Q}} [F(\mathbb S)].
\end{align} 
Theorem~\ref{thm2.1} then follows immediately from~\eqref{2.lower} and~\eqref{2.upper}.

\subsection{Proof of (\ref{2.lower})}\label{sec4}
In this section we prove the inequality \eqref{2.lower}:
\[
\limsup_{n\rightarrow\infty} V_n(F)\leq
\sup_{\mathbb Q\in\mathcal Q} \mathbb{E}_{\mathbb{Q}} [F(\mathbb S)].
\]
Due to the dual description of $V_n(F)$ in Lemma~\ref{lem3.1} we can find for any $n\in\mathbb N$ a probability measure $\mathbb Q_n\in\mathcal{Q}_n$ along with a martingale $M^{(n)}$ close to $S^{(n)}$ in the sense of~\eqref{3.100} such that
\begin{align}\label{4.2}
V_n(F)<\frac{1}{n}+\mathbb E_{\mathbb Q_n} \left[F\left(\mathcal{W}_n(S^{(n)})\right)\right].
\end{align}
We get from \eqref{3.100} that, for every $i=1,\dots, d$ and $n=1,\dots$, the $\QQ_n$-martingales $\tilde{M}^n \set M^{n,i}$ are as
required in the following lemma.
\begin{lem}\label{lem3.2}
For any $n\in\mathbb{N}$ let $\mathbb Q_n$ be a probability measure on $\Omega$ and let $\tilde M^{(n)}=\{\tilde{M}^{(n)}_k\}_{k=0,\dots,n}$ be a positive $\mathbb Q_n$ martingale with $\tilde M^{(n)}_0=1$. Assume that there exists a constant $c>0$ such that
\begin{align}\label{3.2}
|\tilde M^{(n)}_{k+1}-\tilde M^{(n)}_k|\leq \frac{c}{\sqrt n} \tilde M^{(n)}_k.
\end{align}
Then the sequence $\Law(\cW_n(\tilde M^{(n)}) \;|\; \QQ_n)$, $n\in\mathbb N$, is tight on the space $C_d$. Furthermore, for any $m\in\mathbb N$,
\begin{align}
\sup_{n\in\mathbb N}\mathbb E_{\mathbb Q_n}[\max_{0\leq k\leq n} \tilde M^{(n)}_k]^{2m} &<\infty,\label{3.3}\\
\sup_{n\in\mathbb N}\mathbb E_{\mathbb Q_n}[\max_{0\leq k\leq n} |\ln\tilde M^{(n)}_k|]^{2m} &<\infty.\label{3.4}
\end{align}
\end{lem}
\begin{proof}
Following the same arguments as in (4.23) in \cite{K} we obtain \eqref{3.3}. To establish the tightness assertion, we fix $0\leq u<t\leq 1$ and $n\in\mathbb N$. From the Burkholder--Davis--Gundy inequality and \eqref{3.2}--\eqref{3.3} we get
\begin{align*}
\mathbb E_{\mathbb Q_n}&([\cW_n(\tilde M^{(n)})]_t-[\cW_n(\tilde M^{(n)})]_u)^4\\
&= \mathbb E_{\mathbb Q_n}(\tilde M^{(n)}_{[nt]}-\tilde M^{(n)}_{[nu]})^4\\
&\leq \const \mathbb E_{\mathbb Q_n}\left(\sum_{k=[nu]}^{[nt]-1}
\mathbb  E_{\mathbb Q_n}\left((\tilde M^{(n)}_{k+1}-\tilde
  M^{(n)}_k)^2|\mathcal F_k\right)\right)^2\\
&=\const \mathbb E_{\mathbb Q_n}\left(\sum_{k=[nu]}^{[nt]-1}(\tilde M^{(n)}_k)^2
\mathbb E_{\mathbb Q_n}\left(\left(\frac{\tilde M^{(n)}_{k+1}}{\tilde
  M^{(n)}_k}-1\right)^2\bigg|\mathcal F_k\right)\right)^2\\
&\leq
\const (n(t-u))^2\frac{c^4}{n^2}
 \sup_{m\in\mathbb N}\mathbb E_{\mathbb Q_m}\left(\max_{0\leq k\leq
  m} \tilde M^{(m)}_k\right) ^{4}\\
&\leq \const (t-u)^2.
\end{align*}

Thus the sequence $\Law(\cW_n(\tilde M^{(n)}) \; | \; \QQ_n)$, $n\in\mathbb N$, is tight by Kolmogorov's criterion. Finally, similarly to (4.24)--(4.26) in \cite{K}, we obtain \eqref{3.4}.
\end{proof}

From the preceding lemma, we obtain that the sequence $\Law(\cW_n(M^{n}) \;|\; \QQ_n)$, $n\in\mathbb N$, is tight on $C_d$. So there exists a subsequence (which, for ease of notation, we still index by $n$) that converges in law on $C_d$ and from \eqref{3.100} we can conclude that $\Law(\cW_n(S^{(n)})\;|\;\QQ_n)$, $n\in\mathbb N$, converges to the same law.

By Skorohod's representation theorem (see \cite{D}), there is a
probability space
$(\tilde\Omega,\tilde{\mathcal F},\tilde{\mathbb P})$ with processes
$\tilde{S}^{(n)}$ which have the same law as $\cW(S^{(n)})$ under
$\QQ_n$, respectively, and which converge almost surely uniformly on
$[0,1]$ to another continuous process
$M=\{(M^{(1)}_t,...,M^{(d)}_t)\}_{0 \leq t \leq 1}$. Since, for each
$n=1,2,\dots$, $M^{(n)}$ is adapted to the filtration
$(\cF^n_k)_{k=0,\dots,n}$ generated by $S^{(n)}$, we can construct
processes $\tilde{M}^n$ on
$(\tilde\Omega,\tilde{\mathcal F},\tilde{\mathbb P})$ such that, under
$\tilde{\P}$, $(\tilde{S}^{n},\tilde{M}^n)$ has the same joint law as
$(S^{n},M^{(n)})$ under $\QQ_n$. In view of \eqref{3.100} this gives
\begin{align}\label{4.4}
\tilde{M}^{(n)}\rightarrow M \text{ in } C_d \quad  \tilde{\mathbb P}  \text{-a.s.}
\end{align}
By \eqref{4.2}, \eqref{3.3}, and Lemma~\ref{lem:coe},
\begin{align}\label{4.5}
\limsup_{n} V_n(F)\leq\mathbb E_{\tilde{\mathbb P}}[F(M)].
\end{align}
In order to complete the proof of \eqref{2.lower} it now remains to establish that $\Law(M \mid \tilde{\P})$ is contained in $\mathcal{Q}$.

First we show that $M$ is a $\tilde{\P}$-martingale with respect to its own filtration. Fix $0\leq t_1\leq ...\leq t_k\leq u<t\leq 1$ and consider a bounded continuous function $g:\mathbb R^k\rightarrow\mathbb R^d$. From \eqref{3.3} and~\eqref{4.4} we get
\begin{align*}
\mathbb{E}_{\tilde{\mathbb P}}&[( M_t-M_u)' g(M_{t_1},...,M_{t_k})]\\&=
\lim_{n\rightarrow\infty}\mathbb{E}_{\tilde{\mathbb P}}[( \tilde{M}^{(n)}_{[nt]}-\tilde{M}^{(n)}_{[nu]})'
g(\tilde{M}^{(n)}_{[nt_1]},...,\tilde{M}^{(n)}_{[nt_k]}) ]=0,
\end{align*}
where the first equality follows from ~\eqref{4.4} and the second from the fact that, for each $n=1,2,\dots$, $\tilde{M}^{(n)}$ has the same law under $\tilde{\P}$ as $M^{(n)}$ has under $\QQ_n$, and each $M^{(n)}$ is a $\QQ_n$-martingale with respect to the filtration $(\cF^n_k)_{k=0,\dots,n}$ generated by $S^{(n)}$. From standard density arguments we thus obtain the asserted martingale property.

Let us finally verify that $M$ satisfies the volatility condition. Observe that, as a consequence of~\eqref{3.4}, each $M^{(i)}$, $i=1,\dots,d$, is a strictly positive martingale. So let us introduce the continuous stochastic process $N=\{(N^{(1)}_t,...,N^{(d)}_t)\}_{0 \leq t \leq 1}$ by $N^{(i)}_t=\int_{0}^t \frac{dM^{(i)}_u}{M^{(i)}_u}$ for which, clearly,
\begin{align}\label{4.6}
M^{(i)}_t=s_i \exp\left(N^{(i)}_t-\langle N^{(i)}\rangle_t/2\right), \quad i=1,...,d.
\end{align}
We need to show that $\langle N \rangle$ is absolutely continuous almost surely with density $d\langle N \rangle/dt$ taking values in $\Gamma$. For each $k=0,\dots,n$ and $i,j=1,\dots,d$, we define
\begin{align}\label{4.essential}
\alpha^{n,i}_k &\set\frac{1}{\sqrt n}\langle \sigma_i,\xi_k\rangle, \\
A^{n,i}_k&\set \frac{M^{n,i}_k-S^{n,i}_k}{S^{n,i}_k},
           \nonumber\\
N^{n,i}_k&\set\sum_{m=1}^k\Delta N^{n,i}_m, \ttext{where} 
\Delta N^{n,i}_m \set
\frac{M^{n,i}_m-M^{n,i}_{m-1}}{S^{n,i}_{m-1}},\nonumber\\
X^{n,i}_k& \set \sum_{m=1}^k \Delta X^{n,i}_m, \ttext{where} 
\Delta X^{n,i}_m \set
\alpha^{n,i}_m(1+ A^{n,i}_m),\nonumber\\
Y^{n,i,j}_k & \set\sum_{m=1}^k \Delta Y^{n,i,j}_m, \ttext{where}
\Delta Y^{n,i,j}_m \set X^{n,i}_{m-1}\Delta N^{n,j}_m+X^{n,j}_{m-1}\Delta N^{n,i}_m,\nonumber\\
Z^{n,i,j}_k &\set X^{n,i}_k X^{n,j}_k- Y^{n,i,j}_k.\nonumber
\end{align}
We thus obtain the $\RR^d$-valued processes $N^{(n)}=(N^{n,1},...,N^{n,d})$,
$X^{(n)}=(X^{n,1},...,X^{n,d})$, and the matrix-valued processes
$Y^{(n)}=\{Y^{n,i,j}\}_{i,j=1}^d$ and $Z^{(n)}=\{Z^{n,i,j}\}_{i,j=1}^d$. 

By \eqref{2.3}, we have
\begin{align}\label{4.6+}
\Delta N^{n,i}_k-\Delta X^{n,i}_k=\Delta A^{n,i}_k,
\end{align}
so
$|X^{n,i}_k-N^{n,i}_k|=|A^{n,i}_k-A^{n,i}_0| \leq
\frac{\kappa_i}{\sqrt n}$.
This together with \eqref{4.4} and Theorem~4.3 in \cite{DP} gives the
weak convergence
\begin{align}\label{4.7}
\Law\left(\left(S^{(n)}_{[nt]},M^{(n)}_{[nt]},N^{(n)}_{[nt]},X^{(n)}_{[nt]}\right)_{0\leq t \leq 1} \;\middle|\; \P\right) \Rightarrow \Law\left(\left(M,M,N,N\right) \;\middle|\; \tilde{\mathbb P}\right)
\end{align}
on the Skorohod space $(\mathbb D^d)^4$ of all $c\grave{a}dl\grave{a}g$ functions $f:[0,1]\rightarrow\mathbb R^d$ endowed with the Skorohod topology. From \eqref{4.7}, the definitions of the processes $Y,Z$ and
$$
\langle N^{i,j}\rangle_t=N^{(i)}_t N^{(j)}_t-
\int_{0}^t [N^{(i)}_u dN^{(j)}_u+ N^{(j)}_u dN^{(i)}_u],
$$
we can conclude using Theorem 4.3 in \cite{DP} the weak convergence
\begin{align}\label{4.8}
\Law\left(\left(S^{(n)}_{[nt]},M^{(n)}_{[nt]},N^{(n)}_{[nt]},X^{(n)}_{[nt]},Z^{(n)}_{[nt]}\right)_{0
      \leq t \leq 1}
  \;\middle|\; \P\right) \rightarrow \Law\left(M,M,N,N,\langle N\rangle \;\middle|\; \tilde{\mathbb P}\right)
\end{align}
on the space $(\mathbb D^d)^4\times \mathbb D^{d^2}$.

By Skorohod representation theorem, there exists a probability space $(\hat\Omega,\hat{\mathcal F},\hat{\mathbb P})$ such that all the processes $S^{(n)},\dots,N$ in \eqref{4.8} have copies $\hat S^{(n)},\dots,\hat N$ with the same joint distributions on this space and satisfy
\begin{align}\label{4.9}
\left(\hat{S}^{(n)},\hat{M}^{(n)},\hat{N}^{(n)},\hat{X}^{(n)},\hat{Z}^{(n)}\right)\rightarrow \left(\hat M,\hat M,\hat N,\hat N,\langle \hat N\rangle\right) \quad \hat{\mathbb P}\text{-a.s.}
\end{align}
on the space $(\mathbb D^d)^4\times \mathbb D^{d^2}$. Moreover, on $(\hat\Omega,\hat{\mathcal F},\hat{\mathbb P})$, copies in
distribution of the random vectors $\xi_1,...,\xi_n$ can be recovered from $\hat{S}^{(n)}$, $n=1,2,\dots$. These copies now depend on $n$, so we denote them by
$\hat{\xi}^{(n)}_1,\dots, \hat{\xi}^{(n)}_n$.

Next, from \eqref{4.6+} and the summation by parts formula we obtain
\begin{align*}
\hat Y^{n,i,j}_k&=\sum_{m=1}^k [\hat X^{n,i}_{m-1}\Delta \hat X^{n,j}_m+\hat X^{n,j}_{m-1}\Delta \hat X^{n,i}_m]\\
&\qquad +\sum_{m=1}^k [\hat X^{n,i}_{m-1}\Delta \hat A^{n,j}_m+\hat X^{n,j}_{m-1}\Delta \hat A^{n,i}_m]\nonumber\\
&=\hat X^{n,i}_k \hat X^{n,j}_k-\sum_{m=1}^k \Delta \hat X^{n,i}_m\Delta \hat X^{n,j}_m\nonumber\\
&\qquad+\hat X^{n,i}_{k-1}\hat A^{n,j}_k+\hat X^{n,j}_{k-1}\hat A^{n,i}_k
-\sum_{m=1}^{k-1} [\Delta \hat X^{n,i}_m\hat A^{n,j}_m+\Delta \hat X^{n,j}_m \hat A^{n,i}_m]\nonumber\\
&=\hat X^{n,i}_k \hat X^{n,j}_k-
[\sum_{m=1}^k \hat \alpha^{n,i}_m \hat\alpha^{n,j}_m+\hat\alpha^{n,i}_m \hat A^{n,j}_m+\hat\alpha^{n,j}_m \hat A^{n,i}_m]\nonumber\\
&\qquad+O(n^{-1/2})(1+|\hat X^{(n)}_{k-1}|),\nonumber
\end{align*}
where the term $O(n^{-1/2})$ above is uniformly bounded by $c n^{-1/2}$ for some constant $c$. For our matrix-valued process $\hat Z$ this entails
\begin{align}\label{4.10}
\hat Z^{(n)}_k
&=\frac{1}{n}\sum_{m=1}^k\left[\sigma (\hat \xi^{(n)}_m)^{'}\hat \xi^{(n)}_m\sigma^{'}+\sqrt n\sigma (\hat \xi^{(n)}_m)^{'}\hat  A^{(n)}_m+\sqrt n[\sigma(\hat \xi^{(n)}_m)^{'} \hat A^{(n)}_m]^{'}\right]\\
&\qquad+O(n^{-1/2})(1+|\hat X^{(n)}_{k-1}|).\nonumber
\end{align}

Next, we fix $0\leq u<t\leq 1$ and define the random sets of integers
\[
V^{n,i} \set \{[nu]\leq k\leq [nt]  \ \ :\ \ \hat {\xi}^{(n)}_k=v_i\}, \quad
i=1,...,d+1
\]
for each $n$. Denoting by $|V^{n,i}|$ the number of elements in $V^{n,i}$, we then get from \eqref{4.10} that
\begin{align}\label{4.11}
\hat Z^{(n)}_{[nt]}&-\hat Z^{(n)}_{[nu]}\\\nonumber&=\frac{[nt]-[nu]}{n}\sum_{i=1}^{d+1}\frac{|V^{n,i}|}{[nt]-[nu]}\sigma v^{'}_i v_i\sigma^{'}\\
&\qquad+\frac{[nt]-[nu]}{n}\sum_{i=1}^{d+1}\frac{|V^{n,i}|}{[nt]-[nu]}\left(\frac{\sum_{j\in V^{n,i}}\sqrt n\sigma v^{'}_i \hat A^{(n)}_j+\sqrt n[\sigma v^{'}_i \hat A^{(n)}_j]^{'}}{|V^{n,i}|}\right)\nonumber\\
&\qquad+O(n^{-1/2})(1+|\hat X^{(n)}_{[nt]-1}|+|\hat X^{(n)}_{[nu]-1}|).\nonumber
\end{align}
Using Taylor's expansion and (\ref{2.3}) we observe
\begin{align*}
\ln \hat S^{(n)}_{[nt]}-\ln \hat S^{(n)}_{[nu]} &\set (\ln \hat S^{n,1}_{[nt]}-\ln \hat S^{n,1}_{[nu]},...,\ln \hat S^{n,d}_{[nt]}-\ln \hat S^{n,d}_{[nu]})\\
&= \frac{[nt]-[nu]}{\sqrt n}\sum_{i=1}^{d+1}\frac{|V^{n,i}|}{[nt]-[nu]}v_i\sigma^{'} + O(1).
\end{align*}
Here, by \eqref{4.9}, the left side converges $\hat {\mathbb P}$-almost surely, 
so the right side is bounded $\hat \P$-almost surely which implies
\[
\lim_{n\rightarrow\infty}\sum_{i=1}^{d+1}\frac{|V^{n,i}|}{[nt]-[nu]}v_i=0\quad \hat{\mathbb P}\text{-a.s.}
\]
Since $v_1,\dots,v_{d+1}$ are affinely independent with~\eqref{2.1+}, we thus have
\begin{align}\label{4.12}
\lim_{n\rightarrow\infty}\frac{|V^{n,i}|}{[nt]-[nu]}=\frac{1}{d+1}\quad \hat{\mathbb P}\text{-a.s.},\ i=1,\dots,d+1.
\end{align}

From \eqref{2.2+}, \eqref{4.9} and \eqref{4.11}--\eqref{4.12}, we can now conclude that
\begin{align}\label{4.13}
\langle \hat N\rangle_t&-\langle \hat N\rangle_u\\&=
\lim_{n\rightarrow\infty}\hat Z^{(n)}_{[nt]}-\hat Z^{(n)}_{[nu]}\nonumber\\&=(t-u)\left(\sigma\sigma^{'}+
\lim_{n\rightarrow\infty}\frac{1}{d+1}\sum_{i=1}^{d+1}\sum_{j\in V^{n,i}}\frac{\sqrt n\sigma v^{'}_i \hat A^{(n)}_j+\sqrt n[\sigma v^{'}_i  \hat A^{(n)}_j]^{'}}{|V^{n,i}|}\right).\nonumber
\end{align}
We have, by \eqref{2.1+},
\[
\frac{1}{d+1}\sum_{i=1}^{d+1}\sum_{j\in V^{n,i}}\frac{\sqrt n\sigma v^{'}_i \hat A^{(n)}_j}{|V^{n,i}|}=
\sigma\frac{1}{d+1} \sum_{i=1}^{d+1} v^{'}_i\left(\sum_{j\in V^{n,i}}\frac{\sqrt n \hat A^{(n)}_j}{|V^{n,i}|}+(\kappa_1^-,\dots,\kappa_d^-)\right),
\]
where the right side is of the form $\sigma \beta$ for some $\beta\in B$. We conclude that for any $u<t$, $\frac{\langle \hat N\rangle_t-\langle \hat N\rangle_u}{t-u}$ is uniformly bounded and contained in $\Gamma$. As a result, $\langle \hat N \rangle$ has a density taking values in $\Gamma$ $\hat{\mathbb P}\times dt$-a.e., exactly as required.

\subsection{Proof of \eqref{2.upper}}\label{sec5}

\subsubsection{The density argument}

In the first step for the proof of~\eqref{2.upper}, we identify a suitable dense subset of $\cQ$. Lemma~\ref{lem:coe} then allows us to restrict the supremum in \eqref{2.upper} to this subset.

For arbitrary complete probability space $(\Omega^W, \mathcal{F}^{W}, \mathbb P^{W})$ carrying a standard $d$--dimensional Brownian motion $W=\{(W^{(1)}_t,...,W^{(d)}_t)\}_{0 \leq t \leq 1}$, we denote by $\Gamma^W_c$ the set of processes $\alpha$ which are progressively measurable with respect to the augmented filtration generated by $W$ and which are of the form
\begin{align}\label{eq:4}
\alpha_t = \sum_{l=0}^L \alpha_l 1_{(T_{l},T_{l+1}]}(t)
\end{align}
for some $0=T_0<\dots<T_{L+1}=1$  with $\alpha_l$ of the form 
\begin{align}\label{eq:5}
\alpha_l= a_l\left(\int_0^{{t_1^{(l)}}} \langle \alpha, dW \rangle,\dots,\int_0^{t_{n_l}^{(l)}} \langle \alpha, dW \rangle \right)
\end{align}
for some $0\le t_1^{(l)} < \dots < t_{n_l}^{(l)} \le T_l$ and some continuous functions $a_l$ taking values in $\sqrt{\Gamma}$ such that $a_l^2-\epsilon I$ is positive definite for some constant $\epsilon>0$.

\begin{lem}\label{lem3.4}
The set
\begin{align}\label{3.11+}
\left\{\Law\left(\int_0^\cdot\alpha dW\midb \mathbb P^W\right)\midb \mathbb P^W,\alpha\in\Gamma_c^W\right\}
\end{align}
is a dense subset of $\mathcal Q$.
\end{lem}
\begin{proof}
To show that the set in \eqref{3.11+} is a subset of $\mathcal Q$, it suffices to note that, for every $(\Omega^W,\mathcal F^W,\mathbb P^W)$ and $\alpha\in\Gamma^W_c$, we have
\[
\frac{d\langle N\rangle}{dt}=\sum_{l=0}^{L} a^2_l(N_{t^{(l)}_1},...,N_{t^{(l)}_{n_l}})1_{(T_l,T_{l+1}]}, \ \  \mathbb P^W\times dt\text{-a.e.},
\]
where $N_t=\int_0^t\alpha dW$.

Let $(\Omega^W)=C_d$ and let $\Gamma^W_d$ be the set of processes $\theta$ of the form
\[
\theta(W,t)=\sum_{l=0}^L\theta_l(W_{|_{[0,T_l]}}) 1_{[T_l,T_{l+1})}(t)
\]
for some $0=T_0<\dots<T_{L+1}=1$ and some $\theta_l:C([0,T_l])\to\sqrt{\Gamma}$. As shown in the proof of Proposition~3.5. in \cite{DNS}, the convex hull of the laws associated with the stochastic integrals of $\theta\in\Gamma^W_d$ form a dense subset of $\mathcal Q$. It thus suffices to show that elements of $\Gamma^W_d$ are limits in $P^W\times dt$-measure of sequences from $\Gamma^W_c$. Indeed, using boundedness of $\sqrt{\Gamma}$, such convergence implies the convergence of the associated laws for the stochastic integrals.

Each $\theta_l$ is $\mathbb P^W$-almost sure limit of some random variables of the form $\theta_l(\mathcal W_l^{(n)}(\{W_{t^{(l)}_k}\}_{k=1,\dots,n_l}))$, where $0\le t_1^{(l)}<\dots<t_{n_l}^{(l)}\le T_l$ and $\mathcal W_l^{(n)}$ is the linear interpolation operator from the grid $0\le t_1^{(l)}<\dots<t^{(l)}_{n_l}\le T_l$ to $C([0,T_l];\mathbb R^d)$. Defining
\[
\alpha^{(n)}\set\sum_{l=0}^L \theta_l\left(\mathcal W_l^{(n)}\left(\left\{W_{t^{(l)}_k}\right\}_{k=1,\dots,n_l}\right)\right) 1_{[T_l,T_{l+1})},
\]
we see that $\alpha^{(n)}\rightarrow \theta$ $P^W\times dt$ a.e. We conclude the proof by showing that $\alpha^{(n)}$ is of the form \eqref{eq:4}--\eqref{eq:5} for every fixed $n\in\mathbb N$.

We note first that each $\alpha^{(n)}$ is of the form
\[
\alpha^{(n)}=\sum_{l=0}^L\sigma_l\left(W_{t^{(l)}_1},\dots,W_{t^{(l)}_{n_l}}\right) 1_{[T_l,T_{l+1})}
\]
for some continuous $\sqrt{\Gamma}$-valued functions $\sigma_l$. We may assume that $\{t^{(l)}_1,\dots t^{(l)}_{n_l}\}$ is a subset of $\{t^{(l+1)}_1,\dots t^{(l+1)}_{n_{l+1}}\}$ for each $l$ and that each $\sigma_l-\epsilon I$ is positive definite-valued for some $\epsilon>0$; see the beginning of the proof of Proposition 3.5 in \cite{DNS}. 

Let $M_t=\int_0^t \alpha^{(n)} dW$ be the corresponding martingale. We have $M_{t^{(l)}_{j}}=\theta_0W_{t^{l}_j}$ for $t^{(l)}_j\le T_1$ and, for $l\ge 1$ and $t^{(l)}_j\in [T_{l-1},T_{l}]$, 
\[
M_{t^{(l)}_j}=M_{T_{l-1}}+\left (W_{t^{(l)}_j}-W_{T_{l-1}}\right)\sigma_l\left(W_{t^{(l)}_1},\dots,W_{t^{(l)}_{n_{l-1}}}\right).
\]
Using continuity of each $\sigma_l$ and the fact each $\sigma_l$ is strictly positive definite-valued, we get, by an induction argument, that there are continuous functions $\phi^{(l)}_i$ such that $W_{t^{(l)}_j}=\phi^{(l)}_j(M_{t^{(l)}_1},\dots,M_{t^{(l)}_{j}})$. Therefore $\alpha^{(n)}$ are of the form
\[
\alpha^{(n)}=\sum_{l=0}^L a_l\left(M_{t^{(l)}_1},\dots,M_{t^{(l)}_{n_l}}\right) 1_{[T_l,T_{l+1})},
\]
where each $a_l$ is continuous with $a_l-\epsilon I$ positive definite-valued.
\end{proof}

\subsubsection{The multivariate Kusuoka construction}

In view of Lemma~\ref{lem:coe} and Lemma~\ref{lem3.4}, the inequality~\eqref{2.upper} follows if we show that
\begin{align}\label{eq:7}
  \liminf_{n\rightarrow\infty} V_n(F)\geq
  \mathbb{E}_{\mathbb P^W} [F( S^{(\alpha)})].
\end{align} 
for any choice of $\P^W$ and $\alpha \in \Gamma^W_c$ of the form introduced in~\eqref{eq:4} and~\eqref{eq:5}. So fix such a $\P^W$ and $\alpha$, and let us construct, for  $n=1,2,\dots$, processes $M^{(n)}$ which are close to $S^{(n)}$ in the sense of~\eqref{3.100} and which allow for an equivalent martingale measure $\QQ_n \approx \P$ on $(\Omega,\cF_n)$ such that
\begin{equation}\label{eq:16}
\Law(\mathcal{W}_n(S^{(n)}) \;|\; \QQ_n) \Rightarrow \Law(S^{(\alpha)} \;|\; \P^W).
\end{equation}

In our construction it will turn out to be convenient to have the mappings $\Psi$ and $\Phi$ introduced in the following lemma at our disposal.
\begin{lem}\label{lem:pp}
\begin{itemize}
\item[(i)] There exists a measurable mapping $\Psi:\Gamma \to B$ such that $\beta=\Psi(a)$ solves 
\[
a = \sigma \sigma'+\sigma \beta + \beta'\sigma'.
\]
\item[(ii)] There exists a measurable mapping $\Phi:B \to
  \prod_{i=1}^d [-\kappa^-_i,\kappa^+_i]$ such that
\[
v_i\beta+\Phi(\beta)\in\Pi_{j=1}^d [-\kappa_j^-,\kappa_j^+],\quad i=1,\dots,d+1,\ \beta\in B.
\]
\end{itemize}
\end{lem}
\begin{proof}
The mapping $\beta \to\sigma \sigma'+\sigma \beta + \beta'\sigma'$
from $B$ to $\Gamma$ is continuous, so $\Psi$ can be chosen as its
measurable inverse; see Theorem~6.9.7 in \cite{B}. 

By the same theorem, there exists a measurable subset $W$ of the
matrices with columns $w_k \in \prod_{k=1}^d
[0,(\kappa_k^++\kappa^-_k)/(d+1)]$ such that the mapping $w \mapsto
(\sum_{j=1}^{d+1} w_{jk}v_j')_{k=1,\dots,d} \in B$ is a measurable
  bijection from $W$ to $B$ with inverse $\WW(b)=(w_{jk}(b))$. Now, take $$(\Phi(\beta))_k = \sum_{j=1}^{d+1} w_{jk}(\beta)-\kappa^-_k.$$
To prove the property for $\Phi$, it suffices to note that \eqref{2.0+}--\eqref{2.2+} give
\begin{equation}\label{eq:vbeta}
(v_i\beta)_k=(d+1)w_{ik}(\beta) -\sum_{j=1}^{d+1}w_{jk}(\beta).
\end{equation}
\end{proof}

For our construction, we recall that $\alpha$ is piecewise constant as in~\eqref{eq:4}, and we let $a_l$, $l=0,\dots,L$, be the continuous functions of~\eqref{eq:5}. For each fixed $n=1,2,\dots$, we will put
\begin{align}\label{eq:8}
M^{n,i}_k \set S^{n,i}_k(1+A^{n,i}_k), \quad i=1,\dots,d, \text{ for } k=0,\dots,n,
\end{align}
for a suitably constructed process $A^{(n)}=(A^{n,1}_k,\dots,A^{n,d}_k)_{k=0,\dots,n}$. This construction is given by \eqref{eq:10}--\eqref{eq:14} below and Lemmas~\ref{lemma1}--~\ref{lemma3} prove that this construction serves its purpose.\footnote{Our construction is motivated by the arguments we used to prove~\eqref{2.lower} in the previous section which indicate which volatilities can be generated by which price systems close to $S^{(n)}$ as in~\eqref{3.100}; see, in particular, \eqref{4.13} and compare the definition of $A^{(n)}$ in~\eqref{4.essential}.}

We proceed inductively for $l=0,\dots,L$. For $l=0$, we start with
\begin{align}\label{eq:10}
A^{(n)}_0 \set 0
\end{align}
and use the first $[\sqrt{n}]$ periods to ``blend over'' in steps of size $O(1/\sqrt{n})$ from this starting value to our prescription for $A^{(n)}$ of~\eqref{eq:12} below which is chosen to induce the desired volatility structure of the scaling limit on $[0,T_1]$. Specifically, we put 
\begin{align}\label{eq:11}
A^{(n)}_k \set \frac{k}{[\sqrt{n}]} \frac{1}{\sqrt{n}} \left(\xi_k\Psi(a^2_0)+\Phi(\Psi(a^2_0))\right) \quad k=0,\dots,[\sqrt{n}],
\end{align}
and
\begin{align}\label{eq:12}
A^{(n)}_k \set \frac{1}{\sqrt{n}} \left(\xi_k\Psi(a^2_0)+\Phi(\Psi(a^2_0))\right), k=[\sqrt{n}]+1,\dots,[nT_1].
\end{align}
Along with $A^{(n)}$, $S^{(n)}$, and $M^{(n)}$ we define, as in~\eqref{4.essential}, the proxy $N^{(n)}$ for the stochastic logarithm of $M^{(n)}$ by
\begin{align}\label{eq:13}
N^{n,i}_k\set\sum_{m=1}^k\Delta N^{n,i}_m \ttext{where } \Delta N^{n,i}_m \set\frac{M^{n,i}_m-M^{n,i}_{m-1}}{S^{n,i}_{m-1}}.
\end{align}
For notational convenience we shall also define inductively $\alpha^{(n)}_l$, $l=1,\dots,L$, starting with
\begin{equation}\label{eq:15}
\alpha^{(n)}_1 \set a^2_0.
\end{equation}

Proceeding inductively for $l=1,\dots,L$, we again first ``blend over'' in $[\sqrt{n}]$ periods from the previous prescriptions for $A^{(n)}$ to the one of~\eqref{eq:14} below which will turn out to give us the desired volatility structure on $(T_l,T_{l+1}]$:
\begin{align}\label{eq:13}
A^{(n)}_k \set & \left(1-\frac{k-[nT_l]}{[\sqrt{n}]}\right) \frac{1}{\sqrt{n}} \left(\xi_k \Psi(\alpha^{(n)}_{l-1})+\Phi(\Psi(\alpha^{(n)}_{l-1}))\right)\\ \nonumber&\quad +\frac{k-[nT_l]}{[\sqrt{n}]}\frac{1}{\sqrt{n}} \left(\xi_k \Psi(\alpha^{(n)}_{l})+\Phi(\Psi(\alpha^{(n)}_{l}))\right), \\
\nonumber&\qquad k=[nT_l]+1,\dots,[nT_l]+[\sqrt{n}],
\intertext{ and }
\label{eq:14} A^{(n)}_k \set &\frac{1}{\sqrt{n}} \left(\xi_k \Psi(\alpha^{(n)}_{l})+\Phi(\Psi(\alpha^{(n)}_{l}))\right),\\
\nonumber&\qquad k=[nT_l]+[\sqrt{n}]+1,\dots,[nT_{l+1}].
\end{align}
As before we also define our proxy $N^{(n)}$ as in~\eqref{eq:13} and prepare the next inductive step by denoting
\begin{equation}\label{eq:18}
\alpha^{(n)}_{l+1}=a^2_{l+1}\left(N^{(n)}_{[nt^{(l+1)}_1]},\dots,N^{(n)}_{[n t^{(l+1)}_{n_{l+1}}]}\right).
\end{equation}
The following two lemmas motivate our construction.
\begin{lem}\label{lemma1}
For each $n=1,2,\dots$, the process $A^{(n)}$ of~\eqref{eq:10}--\eqref{eq:14} takes values in $\frac{1}{\sqrt{n}}\prod_{i=1}^d[-\kappa^-_i,\kappa^+_i]$. In particular, our choice $M^{(n)}$ of~\eqref{eq:8} is close to $S^{(n)}$ in the sense that~\eqref{3.100} holds true. 
\end{lem}
\begin{proof}
Using the properties of $\Phi$ and $\Psi$ from Lemma~\ref{lem:pp}, it is straight-forward to verify that $A^{(n)}$ takes values in $\frac{1}{\sqrt{n}}\prod_{i=1}^d[-\kappa^-_i,\kappa^+_i]$ $P$-a.s.. Thus the property \eqref{3.100} for $M^{(n)}$ follows directly from \eqref{eq:8}.
\end{proof}

\begin{lem}\label{lemma2}
Under Assumption~\ref{2.9}, each of the processes $M^{(n)}$, $n=1,2,\dots$, of~\eqref{eq:8} allows for a martingale measure $\QQ^n \approx \P$ on $(\Omega,\cF_n)$.
\end{lem}
\begin{proof}
Due to~\eqref{eq:13} it suffices to show that there is an equivalent martingale measure $\QQ_n$ for $N^{(n)}$. Using \eqref{4.6+} and our construction of $A^{(n)}$, we can write
\begin{align*}
\Delta N^{(n)}_k &= \Delta A^{(n)}_k + \frac{1}{\sqrt{n}}\xi_k \sigma'(1+ A^{(n)}_k)\\
 &= \frac{1}{\sqrt{n}} (\xi_k (\beta^{(n)}_{k}+\sigma')-\xi_{k-1} \beta^{(n)}_k)+O(1/n)
\end{align*}
for some $\cF_{k-1}$-measurable $\beta^{(n)}_k$ with values in $B$ of~\eqref{eq:1}. This reveals that, for sufficiently large $n$, we can find a conditional probability $\QQ_n[\dots \;|\; \cF_{k-1}]$ equivalent to $\P[\dots\;|\;\cF_{k-1}]$, if the support of $\xi_{k-1} \beta$ is contained in the interior of the convex hull of the support of $\xi_k (\beta+\sigma')$ uniformly in $\beta$, i.e., if there exists $\epsilon>0$ such that for any $\beta \in B$ and $w\in\mathbb R^d$ with $|w|<\epsilon$,
\[
v_i \beta \in \interior \conv(v_1(\beta+\sigma')+w,\dots,v_{d+1}(\beta+\sigma')+w)\quad\forall i=1,\dots,d+1.
\]
Using affine independence of $v_1,\dots,v_{d+1}$ and \eqref{2.0+}, it is elementary to verify that $w\in\interior\conv\{v_1,\dots,v_{d+1}\}$ if and only if
\[
\langle w,v_i\rangle>-1\quad\forall\ i=1,\dots,d+1.
\]
Thus, under Assumption~\ref{2.9}, each $A_i(\beta):=v_i\beta(\beta+\sigma')^{-1}$ is a continuous mapping from $B$ to $\interior \conv(v_1,\dots,v_{d+1})$. Using compactness of $B$ and continuity of each $A_i$, there exists $\tilde\epsilon>0$ such that $A_i(\beta)+w\in\interior \conv(v_1,\dots,v_{d+1})$ for all $i=1,\dots,d+1$, $\beta\in B$ and $w\in\mathbb R^d$ with $|w|<\tilde\epsilon$. Using compactness of $B$ again, this implies the existence of the required $\epsilon>0$. 
\end{proof}

\begin{lem}\label{lemma3}
The weak convergence~\eqref{eq:16} holds true for $M^{(n)}$ and $\QQ_n$, $n=1,2,\dots$, as in Lemma~\ref{lemma2}.
\end{lem}
\begin{proof}
Using Lemma~\ref{lemma1} and Lemma~\ref{lemma2}, we can repeat the proof of \eqref{2.lower} from Lemma~\ref{lem3.2} up to \eqref{4.13}. In particular, there exists a probability space $(\hat\Omega,\hat{\mathcal F},\hat{\mathbb P})$ on which the processes $A^{(n)},N^{(n)},\dots$ have copies $\hat A^{(n)},\hat N^{(n)},\dots$,
\begin{align}\label{eq:wc}
\left(\hat{S}^{(n)},\hat{M}^{(n)},\hat{N}^{(n)},\hat{X}^{(n)},\hat{Z}^{(n)}\right)\rightarrow \left(\hat M,\hat M,\hat N,\hat N,\langle \hat N\rangle\right) \quad \hat{\mathbb P}\text{-a.s.}
\end{align}
$\hat N$ is a $\mathbb P$-martingale, and
\begin{align*}
\langle \hat N\rangle_t-\langle \hat N\rangle_u=(t-u)\left(\sigma\sigma^{'}+
\lim_{n\rightarrow\infty}\sum_{i=1}^{d+1}\sum_{j\in V^{n,i}}\frac{\sqrt n\sigma v^{'}_i \hat A^{(n)}_j+\sqrt n[\sigma v^{'}_i  \hat A^{(n)}_j]^{'}}{(d+1)|V^{n,i}|}\right),\nonumber
\end{align*}
where $V^{n,i} =\{[nu]\leq k\leq [nt]  \ \ :\ \ \hat {\xi}^{(n)}_k=v_i\}$.

Fix $l\le m$ and $T_l<u<t<T_{l+1}$. Assume from now on that $n$ is sufficiently large so that we have $[nu]> [nT_l]+[\sqrt n]$ and, in particular, $\hat A^{(n)}_k=\frac{1}{\sqrt{n}} \left(\hat\xi_k \Psi(\hat\alpha^{(n)}_{l})+\Phi(\Psi(\hat\alpha^{(n)}_{l}))\right)$ for any $[nu]\leq k\leq [nt]$. We first observe that
\begin{align*}
\sum_{i=1}^{d+1}\sum_{j\in V^{n,i}}\frac{\sqrt{n}\sigma v^{'}_i \hat A^{(n)}_j}{(d+1)|V^{n,i}|} &= \sum_{i=1}^{d+1}\sum_{j\in V^{n,i}}\frac{\sigma v^{'}_i (v_i \Psi(\hat\alpha^{(n)}_{l})+\Phi(\Psi(\hat\alpha^{(n)}_{l})))}{(d+1)|V^{n,i}|} \\
&= \sigma\Psi(\hat\alpha^{(n)}_l),
\end{align*}
where the last equality follows from \eqref{2.1+}--\eqref{2.2+}. Therefore
\begin{align*}
\langle \hat N\rangle_t-\langle \hat N\rangle_u&=(t-u)\lim_{n\rightarrow\infty}\left(\sigma\sigma^{'}+\sigma\Psi(\hat\alpha^{(n)}_l)+[\sigma\Psi(\hat\alpha^{(n)}_l)]')\right)\\
&=(t-u) \lim_{n\rightarrow\infty}\hat\alpha^{(n)}_l\\
&=(t-u)a^2_l(\hat N_{t^{(l)}_1},...,\hat N_{t^{(l)}_{k_i}}),
\end{align*}
where the second equality follows from the property of $\Psi$ in Lemma~\ref{lem:pp} and the last from \eqref{eq:18} and from the fact that $\hat N^{(n)}\rightarrow \hat N$ $\hat{\mathbb P}$-a.s. We conclude that
\[
\frac{d\langle \hat N\rangle}{dt}=a^2_l(\hat N_{t^{(l)}_1},...,\hat N_{t^{(l)}_{k_i}}), \ \ t\in (T_l, T_{l+1}).
\]
Hence, the quadratic variation of $\hat{N}$ is piecewise constant. It
follows by induction over $l=0,\dots,L$ that the law of
$(\hat{N}_t)_{0 \leq t \leq T_l}$
coincides with the law of $(\int_0^t \alpha \,dW)_{0 \leq t \leq T_l}$
under $\P^W$.
\end{proof}

\section{Illustrations and ramifications}\label{sec6}\setcounter{equation}{0}

\subsection{Sufficient condition for Assumption~\ref{asm2.1}}

The following lemma gives a sufficient condition for Assumption
\ref{asm2.1}. The condition that we give depends only on the reference
volatility matrix and on the transaction cost coefficients, but not on
the choice of $v_1,...,v_{d+1}$.
\begin{lem}\label{lem6.1}
Assumption \ref{asm2.1} is satisfied whenever $|x(\sigma^{'})^{-1}|<\frac{1}{2\sqrt d}$ for all $x\in \prod_{j=1}^d [-\kappa^-_j-\kappa_j^+,\kappa^-_j+\kappa_j^+]$.
\end{lem}
\begin{proof}
  Let $\beta\in B$ and denote $a=\beta(\sigma^{'})^{-1}$. From
  \eqref{eq:vbeta} we get that
  $v_i \beta\in \prod_{j=1}^d
  [-\kappa^-_j-\kappa_j^+,\kappa^-_j+\kappa_j^+]$,
  so our condition implies that for any $i$,
  $|v_i a|<\frac{1}{2\sqrt d}$. Thus, for any $i,j$, we have
  $|\langle v_i a,v_j\rangle|\leq |v_j| \frac{1}{2\sqrt d}=1/2$. This
  means that, for any $i$,
  $-v_i a\in \frac{1}{2} \interior \conv(v_1,...,v_{d+1})$ and so, for any
  $k$, $v_i (-a)^k\in\frac{1}{2^k} \interior \conv(v_1,...,v_{d+1}).$
  Thus we have
\[
v_i \beta (\beta+\sigma^{'})^{-1}v^{'}_j=v_i a (I+a)^{-1}v^{'}_j=\sum_{n=1}^\infty  v_i (-a)^n v^{'}_j>-\sum_{n=1}^\infty \frac{1}{2^n}=-1
\]
for any $i,j$.
 \end{proof}




\subsection{Assumption~\ref{asm2.1} is  essential for Theorem~\ref{thm2.1}}

Assumption~\ref{asm2.1} cannot be omitted from our main result, Theorem~\ref{thm2.1}, as demonstrated in the next example.

\begin{example}

  Let $d=2$, $\sigma=I$, $\kappa^{+}_1=\kappa^{-}_1=\kappa^{-}_2=0$
  and $\kappa^{+}_2= 3\sqrt{2}/4$. Next, let
  $v_1=(0,\sqrt 2)$, $v_2=(\frac{\sqrt 6}{2},-\frac{\sqrt 2}{2})$ and
  $v_3=(-\frac{\sqrt 6}{2},-\frac{\sqrt 2}{2})$.

Observe that $\kappa_2(v_2+v_3)=(0,-3/2)$ and so
\[
\beta \set \left[ {\begin{array}{cc}
 0 & 0 \\
 0 & -1/2 \\
\end{array} } \right]\in B
\]
violates condition~\eqref{2.9} of Assumption \ref{asm2.1}, since $v_1 \beta (\beta+\sigma^{'})^{-1}v^{'}_1=-2$.
Moreover, we have
\[
  a\set \sigma \sigma'+\sigma \beta + \beta'\sigma'= \left[ {\begin{array}{cc}
   1 & 0 \\
   0 & 0 \\
  \end{array} } \right]\in\Gamma.
\]

We will show below that, in any model which emerges as the weak
limiting point of a sequence
$\Law(\mathcal{W}_n(S^{(n)}) \;|\; \QQ_n)$ with $\QQ_n$ as in
Lemma~\ref{lem3.1}, the second asset's  volatility is uniformly
bounded away from zero. As a consequence, for a payoff
$F(S)=f(S^{(2)}_T)$ with any strictly concave function $f$, we can then
follow the arguments of Section~\ref{sec4} leading to~\eqref{4.5} to
conclude the first estimate in the following contradiction to
assertion~\eqref{eq:155} of Theorem~\ref{thm2.1}:
$$
\limsup_n V_n(F) \leq \E_{\tilde{\P}} [F(M))] <
f(s_2) \leq \sup_{\QQ \in \mathcal{Q}} \E_{Q} F(S).
$$
Here $M$ is a $\tilde{\P}$-martingale starting in $M_0=s_2$ whose
second component has volatility bounded away from zero. The second
estimate is then immediate from Jensen's inequality and the last from
the fact that $a$ belongs to $\Gamma$.

To verify the volatility bound for weak limits as discussed above,
recall the notation from Section~\ref{sec4}, in particular
\[
V^{n,1}=\{0\leq k\leq n  \ \ :\ \ \xi^{(n)}_k=v_1\}.
\]
Observe from~\eqref{4.6+} and
$A^{n,2}_k \in [-\kappa^-_2,\kappa^+_2]=[0,\kappa^+_2]$ that, for any
$k\in V^{n,1}$,
\[
\Delta N^{n,2}_k=\Delta X^{n,2}_k+\Delta A^{n,2}_k =
\frac{1}{\sqrt n} v_1^{(2)}(1+A^{n,2}_k)+\Delta A^{n,2}_k\geq \frac{1}{\sqrt n}(\sqrt 2-\kappa_2^+)= \frac{1}{\sqrt {8 n}}.
\]
Thus, by applying Theorem~4.3 in \cite{DP} and \eqref{4.12}, we get
for any $0 \leq s < t \leq T$:
\[
\langle N\rangle^{22}_t-\langle N\rangle^{22}_s=\lim_{n\rightarrow\infty}\sum_{m=[ns]}^{[nt]} [\Delta N^{n,2}_m]^2\geq (t-s)\lim_{n\rightarrow\infty}\frac{|V^{n,1}|}{8 n}=\frac{t-s}{24}.
\]
We conclude that $d\langle N\rangle^{22}/dt \geq 1/24$ as claimed.
\end{example}

\subsection{Dependence on affine basis in discrete-time reference model}

Our next example demonstrates that the limiting superhedging prices depend on the choice of $v_1,\dots,v_{d+1}$.
\begin{example}
Let $d=2$, $\sigma=I$, $\kappa_1^+=\kappa_1^-=\kappa_2^+=0$ and $0<\kappa^+_2<\frac{1}{2\sqrt 2}$. Observe that the conditions of Lemma \ref{lem6.1} are satisified so that Assumption~\ref{asm2.1} is satisfied for all choices of $v_1,\dots,v_{d+1}$. Consider the Margrabe's exchange option with the payoff  $F(S)=(S^{(1)}_T- S^{(2)}_T)^{+}$. Clearly, if
\[
a=\left[ {\begin{array}{cc}
a_{11} & a_{12} \\
a_{21} & a_{22} \\
\end{array} } \right]=Y Y^{'}.
\]
for some $\mathbb R^{2\times 2}$-matrix $Y$, then the Euclidean distance between the rows of $Y$ equals to  $\sqrt{a_{11}+a_{22}-a_{12}-a_{21}}$. Thus from Section 6 in \cite{RV} and Theorem \ref{thm2.1} it follows that
\begin{align}\label{7.1}
\lim_{n\rightarrow\infty} V_n(F)= C\left(s_1,s_2,\sup_{a\in\Gamma}\sqrt{a_{11}+a_{22}-a_{12}-a_{21}}\right),
\end{align}
where $C$ is the Black--Scholes (one dimensional) price of a call option with strike $K$, maturity $1$, interest rate $0$, volatility $\nu$, and initial stock price $s$, i.e., 
\[
C(s,K,\nu)=sN\left(\frac{\ln s-\ln K}{\nu}+\frac{\nu}{2}\right)-KN\left(\frac{\ln s-\ln K}{\nu}-\frac{\nu}{2}\right),
\]
where $N$ denotes the cumulative distribution function of the standard normal distribution.

Here $\beta=[\beta_1\ \beta_2]\in B$ satisty $\beta_1=0$ and $\beta_2=\sum_{j=1}^{d+1} w_{j2}v_j'$ for some $w_{j2}\in[0,\kappa_2^+/3]$, so $\Gamma$ is the set of all matrices of the form
\begin{align*}
a=\left[ {\begin{array}{cc}
1 & \beta^{(1)}_2 \\
\beta^{(1)}_2 & 1+2 \beta^{(2)}_2  \\
\end{array} } \right] , \ \ \beta\in B.
\end{align*}
Thus it is elementary to verify that
\[
\sup_{a\in\Gamma}\sqrt{a_{11}+a_{22}-a_{12}-a_{21}}=\sqrt{2+\frac{2\kappa_2^+}{3}\sum_{j=1}^{d+1}\max\{v_j^{(1)}-v_j^{(2)},0\}}.
\]
We conclude from \eqref{7.1} that $\lim_{n\rightarrow\infty} V_n(F)$ depends on the choice of the base.
\end{example}

\subsection{Multivariate Cox-Ross-Rubinstein yields trivial
  superreplication prices}

Instead of our choice of discrete-time reference model from~\cite{H},
a tempting alternative would be to consider $d$ independent copies of
a Cox-Ross-Rubinstein model. Let us illustrate in our final example
that scaling limits with such a reference model can be trivial even
without transaction costs.

\begin{example} 
Let $d=2$
and
\[
  \sigma=\left[ {\begin{array}{cc}
   1 & 0 \\
   0 & 1 \\
  \end{array} } \right].
\]
The alternative approach for binomial approximations would then be to
define the $n$-step model by
\begin{equation*}
\begin{split}
S^{n,i}_k=s_i \prod_{m=1}^k\left(1+
\sqrt{\frac{1}{n}}\xi^{(i)}_m\right)
\end{split}
\end{equation*}
where $\xi^{(i)}_k$, $i=1,2$, $k=1,...,n$ are i.i.d. symmetric random
variables which take on the values $\pm 1$ under $\P$.

Let us show that even without transaction costs
(i.e. $\kappa_1=\kappa_2=0$) the limit of the super--replication
prices in this model can be trivial, while the Black--Scholes price is
not. Consider, e.g., the payoff $F(S)=S^{(1)}_T\wedge
S^{(2)}_T$.
Assume for simplicity that $s_1=s_2$. Then clearly the Black--Scholes
price of this claim is equal to $\mathbb E_{\mathbb P^W}
F(S^I)<s_1$. In the alternative reference model, however, the
super-replication price can be estimated by considering the martingale
measure $\QQ_n$ under which the $\xi^{(1)}_k$s are i.i.d. and coincide with the
respective $\xi^{(2)}_k$s. Hence, we have
$S^{(1)}=S^{(2)}$ almost surely and, thus,
$$V_n\geq \mathbb E_{\mathbb Q_n} F(S)= \mathbb E_{\mathbb Q_n} S^{n,1}_n=s_1.$$
\end{example}

\bibliographystyle{spbasic}

\begin{thebibliography}{}

\bibitem{B} V.I.Bogachev,
{\em Measure theory. {V}ol. {II},}
Springer-Verlag. (2007).

\bibitem{BT}
B.~Bouchard and N.~Touzi,
{\em Explicit solution of the multivariate super-replication
problem under transaction costs,} Ann.~Appl.~Prob., {\bf 10},
685--708, (2000).

\bibitem{CPT} J.~Cvitanic, H.~Pham and N.~Touzi,
{\em A closed-form solution to the problem of superreplication under transaction costs,}
Finance and Stoch., {\bf 4}, 35--54, (1999).

\bibitem{CRR}  J.C.~Cox, A.R.~Ross and M.~Rubinstein,
{\em Option pricing: A simplified approach,}
J. Financ. Econom., {\bf 7}, 229--263, (1976).

\bibitem{D} RM.Dudley, {\em Distances of Probability Measures and Random Variables,}
Ann. Math. Statist., {\bf 39}, 1563-1572, (1968).

\bibitem{DC}
M.H.A.~Davis and J.M.C.~Clark,
{\em A note on super--replicating strategies,}
Philos. Trans. Roy. Soc. London Ser. A., {\bf 347}, 485--494, (1994).

\bibitem{DNS} Y.Dolinsky, M.Nutz and H.M.Soner,
{\em Weak Approximations of $G$--Expectations,}
Stochastic Processes and their Applications., {\bf 2}, 664--675, (2012).

\bibitem{DP} D.Duffie and P.Protter, {\em From Discrete to Continuous Time Finance: Weak Convergence of the Financial Gain Process,}
Math. Finance., {\bf 2}, 1--15, (1992).

\bibitem{DS} Y.Dolinsky and H.M.Soner,
{\em Duality and Convergence for Binomial Markets with Friction,}
Finance and Stochastics., {\bf 17}, 447--475, (2013).

\bibitem{GRS} P.~Guasoni, M.~Rasonyi and W.~Schachermayer,
{\em Consistent Price Systems and Face-Lifting Pricing under Transaction Costs,}
Ann.~Appl.~Prob., {\bf 18}, 491--520, (2008).

\bibitem{H} H.He,
{\em Convergence from discrete to continuous time
contingent claim prices,}
Rev. Financial Stud.
{\bf 3,} 523--546. (1990).

\bibitem{JK}
E.Jouini and H.Kallal,
{\em Martingales and Arbitrage in Securities Markets with Transaction Costs,}
Journal of Economic Theory.
{\bf 66,} 178--197. (1995).

\bibitem{JLR}
P.~Jakubenas, S.~Levental and M.~Ryznar,
{\em The super-replication problem via probabilistic methods,}
Ann.~Appl.~Prob., {\bf 13}, 742--773, (2003).

\bibitem{K} S.Kusuoka, {\em Limit Theorem on Option Replication Cost with Transaction Costs,} Ann. Appl. Probab. {\bf 5},
198--221, (1995).

\bibitem{LS} S.~Levental and A.V.~Skorohod,
{\em On the possibility of hedging options in the presence of transaction costs,}
Ann.~Appl.~Prob., {\bf 7}, 410-443, (1997).

\bibitem{P1} S.~Peng, {\em G--expectation, G--Brownian motion
and related stochastic calculus of It\^{o} type,} Stochastic
Analysis and Applications, volume 2 of Abel Symp.,
541--567, (2007).

\bibitem{P2} S.~Peng, {\em Multi--dimensional $G$--Brownian motion and
related stochastic calculus under $G$--expectation.,}
Stochastic Processes and Applications, {\bf 12}, 2223--2253, (2008).

\bibitem{RV}
S. Romagnoli and T. Vargiolu,
{\em Robustness of the Black-Scholes approach in the case of options on several assets,}
Finance and Stochastics. {\bf 4}, (2000).

\bibitem{SSC} H.M.~Soner, S.E.~Shreve and J.~Cvitanic,
{\em There is no nontrivial hedging portfolio for option pricing with transaction costs,}
Ann.~Appl.~Prob., {\bf 5},  327--355, (1995).


\end{thebibliography}

\end{document}